\documentclass{amsart}
\usepackage{amssymb,amscd,amsthm}
\usepackage{latexsym}

\numberwithin{equation}{section}

\newtheorem*{acknowledgments}{Acknowledgments}
\newtheorem{theorem}{Theorem}
\newtheorem{remark}{Remark}
\newtheorem{lemma}{Lemma}

\newtheorem{definition}{Definition}
\begin{document}
\title{Random Sturm Liouville Operators}

\author[R.\ del Rio]{Rafael del Rio}

\address{IIMAS, UNAM, 04510~M\'exico DF, M\'exico}

\email{delrio@leibniz.iimas.unam.mx}

\date{\today}


\begin{abstract}
 Selfadjoint Sturm-Liouville operators $H_\omega$ on $L_2(a,b)$ with random potentials are considered 
and it is proven, using positivity conditions, that for almost every $\omega$ the operator $H_\omega$ does not share eigenvalues with a broad family of random operators and 
in particular with  operators generated in the same way as   $H_\omega$ but in 
 $L_2(\tilde a,\tilde b)$ where $(\tilde a,\tilde b)\subset(a,b)$. 
\end{abstract}

\maketitle

\section{Introduction}

 This note is about eigenvalues of selfadjoint operators in Hilbert space $L_2(a,b)$,
$( -\infty  \leq a,b\leq \infty)$  which are generated by expressions of the 
form $$H_\omega =- {d^2\over dx} +v(x) +\sum_{n\in I\subseteq\mathbb{Z}}\omega(n)f_n(x)$$
where $\omega(n)$ are independent random variables with continuous
(may be singular) probability distributions. The operators considered here are in a sense more general than the  Anderson type model,
 which normally requires the random variables to be identically distributed. 

For  Schr\"odinger and Jacobi operators
with ergodic potentials it is well known that the probability of a given $\lambda\in\mathbb{R}$
to be an eigenvalue is zero (see \cite{cycon}). The same holds for the operators $H_\omega$ mentioned above, when $\lambda$ is independent of $\omega$. 
In case  $\lambda$ depends on $\omega =\{\omega (n)\}_{n \in I}$, the number $\lambda(\omega)$ could be an eigenvalue of $H_\omega$ for every $\omega$.
 Nevertheless it will be proven below that if the point $\lambda$  is independent of one entry $\omega(n_0)$, but depends otherwise arbitrarily on the sequence $\omega$, then $\lambda(\omega)$ is almost surely
not an eigenvalue of  $H_\omega$, provided  $f_{n_0}$ is positive in an interval.  
From here it will follow that  $H_\omega$ does not share eigenvalues with a large class of random operators.

 
There are interesting relations  between the spectra of selfadjoint operators generated in different Hilbert spaces through the same fixed differential expression.  The spectrum of an operator $H$ defined on an $ L_2(a,b)$, for example,  may be approximated by the spectrum of regular
problems $H_n$  generated through the same differential expression as $H$  but defined  on $L_2(a_n,b_n)$ where $(a_n,b_n)$ are subintervals of $(a,b)$. In fact  eigenvalues of $H$ are limits of eigenvalues of $H_n$ if $a_n\rightarrow a$ and $b_n\rightarrow b$ (see \cite{W1}).
 Here i show in a probabilistic context, that eigenvalues of  operators defined in the same way as $H_\omega$, just in a different $L_2$ spaces, are not eigenvalues of $H_\omega$ almost surely. In particular  the  regular operators $H_{\omega n} $ mentioned above, only share eigenvalues with $H_\omega $ in a set of $\omega 's$  with  zero measure. 


 This work is organized as follows. In section \ref{prel} a generalization of a classical result is presented, about  the set of parameters
$\lambda $ where an operator $H_\lambda $ has a given eigenvalue. In  section \ref{main result} the operators are introduced in a probabilistic setting and 
it is proven that they do not share eigenvalues with other families of random operators. In section \ref{appl}  particular cases are mentioned, where previous theorems can be applied.



\section{Preliminaries}\label{prel}

Let us consider Sturm-Liouville differential expressions 
\begin{equation}\label{A}
(\tau u)(x)= {d^2\over dx}u(x) +q(x)u(x)\quad \mbox{where} \quad x\in (a,b)
\end{equation} 
 and $q$ is a real valued locally integrable function. assume
 the limit point case occurs at $a$ or that $\tau$ is regular by $a$ and the same possibilities for the point $b$. For these concepts see \cite{w}.
 Let $$D_\tau=\{u\in{L_2(a,b)}:u,u\prime\quad\mbox{absolutely continuous in} \quad(a,b),\quad \tau f \in L_2(a,b)\}$$ and consider the selfadjoint realization $H$ of $\tau$ on $L_2(a,b)$ defined as 
\begin{equation}\label{B}
Hu=\tau u
\end{equation} 
$$D(H)=\left\{u\in D_\tau:
\begin{aligned}u(a)cos(\alpha)+u\prime(a)sin(\alpha)=0\quad \mbox{in case}\;\tau\;\mbox{regular by}\; a\\ u(b)cos(\beta)+u\prime(b)sin(\beta)=0\quad\mbox{in case}\;\tau\;\mbox{regular by}
 \;b\end{aligned}\right\}$$
 
  Let  $f:(a,b)\longrightarrow\mathbb{R}$ be an integrable function such that 
\begin{eqnarray}
\label{f}f(x)\left\{\begin{aligned} &>0 \quad\mbox{for almost all}\quad x\in (c,d) \\&=0\quad\mbox{otherwise}\end{aligned}\right.
\end{eqnarray} 
where $[c,d]\subset(a,b)$. In the expression for $\tau$, set $q(x)=v(x)+\lambda f(x)$ with $\lambda \in \mathbb{R}$ and $v$ is a fixed locally integrable function. Denote the operator generated by this potential with $H_\lambda$. For $\gamma,\theta \in [0,\pi)$ let us define the regular operator $H_\lambda^{\theta\gamma}$ 
in $L_2(c,d)$ as the operator generated through the same differential expression as $H_\lambda$ and the boundary conditions 
\begin{eqnarray}\label{bc1}u(c)cos(\theta)+u\prime(c)sin(\theta)=0\\\label{bc2} 
u(d)cos(\gamma)+u\prime(d)sin(\gamma)=0
\end{eqnarray}
For fixed $E\in\mathbb{R}$ define$$A(E):=\{\lambda \in\mathbb{R}:H_\lambda\varphi=E\varphi\;\mbox{for some}\;\varphi\not\equiv 0 \}$$
The eigenvalues of an operator $H$ will be denoted by $\sigma_p(H)$ that is  $$\sigma_p(H):=\{r\in\mathbb{R}: H\varphi=r\varphi\;\mbox{for some}\;\varphi\not\equiv 0\}$$

 I shall need the following 

\begin{lemma}\label{lema1}
Assume $A(E)\not=\emptyset$. There exist $\theta_0,\gamma_0 \in[0,\pi)$ such that $$\lambda\in A(E)\Longleftrightarrow E\in\sigma_p(H_\lambda^{\theta_0\gamma_0})$$
\begin{proof}
  Take $\lambda_0\in A(E)$. Then $H_{\lambda_0}\varphi=E\varphi$ for some $\varphi\in D(H_{\lambda_0})$. Let
 $\theta_0 \,, \gamma_0 \in [0,\pi)$ be the points where
\begin{eqnarray}\label{cond.1}
\varphi(c)cos(\theta_0)+\varphi\prime(c)sin(\theta_0)=0\\\label{cond.2}  
\varphi(d)cos(\gamma_0)+\varphi\prime(d)sin(\gamma_0)=0
\end{eqnarray} 
hold.

$\Rightarrow )$  If $\lambda=\lambda_0$  the assertion follows straightforward since $H_{\lambda_0}^{\theta_0\gamma_0}\varphi=E\varphi $. If  $\lambda \in A, \;\;\lambda\not=\lambda_0$, then $H_\lambda\psi=E\psi$ for some $\psi\not\equiv 0$. 
Therefore, there exist $\theta,\;\gamma \in [0,\pi)$ such that boundary conditions  \ref{bc1} and \ref{bc2} hold for $\psi$.
If we prove that $\theta=\theta_0$ and $\gamma=\gamma_0$  then $H_\lambda^{\theta_0\gamma_0}\psi=E\psi $ and therefore $E\in\sigma_p(H_\lambda^{\theta_0\gamma_0})$. Let us prove $\gamma=\gamma_0$. The proof for $\theta$ is analogous.

a) Assume limit point case (lpc) at b. . If $\gamma\not=\gamma_0$ then $\psi $ and $\varphi $ are linearly independent in $[d,b)$
since $\psi =k\varphi $ for a Constant $k$ would imply $\gamma=\gamma_0$. Therefore any solution $u$ of $\tau u =Eu$
in $[d,b)$ can be written as $u=c_1\varphi +c_2\psi $. Since $\psi $ and $\varphi $ are in $L_2$ ,then  $u$ is in $L_2$ too and we get a contradiction to the lac condition.  This argument is analogous to the one given in \cite{dr} p. 429 .

b) Assume b is regular. Then the vectors $\varphi$ and $\psi$ defined above satisfy $$u(b)cos(\beta)+u\prime(b)sin(\beta)=0$$ and the equation$$\tau u =Eu \quad\mbox{for}\quad x\in[d,b]$$ because f vanishes outside $(c,d)$ (condition \ref{f}).
 Therefore  $W(\varphi , \psi)(b)=0$ where $W$ denotes the Wronskian. This implies linear dependence between $\varphi$ and $\psi$  and therefore $\gamma=\gamma_0$.

$\Leftarrow )$ The proof is similar to the one given in \cite{dr} p. 428 . The details are left to the reader. This direction will not be needed in the rest of this paper.
\end{proof}

\end{lemma}

 The following theorem presents a generalization of a classical result. See theorem 8.3.1 of \cite{Atkinson} or
theorem 8.26 of \cite{w}.

\begin{theorem}\label{thm1}
 For any fixed $E\in \mathbb{R}$, the set $A(E)$ is at most countable.
\end{theorem}
 \begin{proof}
 For fixed $E$ there is at most a countable set of $\lambda$ for which  $E\in\sigma_p(H_\lambda^{\theta_0\gamma_0})$, see for example theorem 8.3.1  \cite{Atkinson}. The result follows then from Lemma\ref{lema1}.
 \end{proof}

\section{main result}
\label{main result}
Fix $n_1,n_2$ in $\mathbb{Z}\cup\{\infty\}\cup\{-\infty\}$, $n_1<n_2$ and define an interval $I$ of $\mathbb{Z}$ as follows$$I=\{n\in\mathbb{Z}:n_1<n<n_2\}$$

The linear space of real valued sequences $\{\omega(n)\}_{n\in I}$ will be denoted by $\mathbb{R}^I$. Let us introduce a measure in $\mathbb{R}^I$ as follows. 
 Let
$\{p_n\}_{n\in I}$ be a sequence of arbitrary probability measures on
$\mathbb{R}$  and consider the product measure $\mathbb{P}
=\mathop{\times}_{n\in I}p_n$ defined on the product $\sigma$-algebra
$\mathcal{F}$ of $\mathbb{R}^I$ generated by the cylinder sets, i.\,e,
by sets of the form $\{\omega:\omega(i_1)\in A_1,\dots,\omega(i_n)\in
A_n\}$ for $i_1,\dots, i_n\in I$, where $A_1,\dots,A_n$ are Arel's sets in
$\mathbb{R}$. A measure space 
$\Omega=(\mathbb{R}^I,\mathcal{F},\mathbb{P})$ is thus constructed.

For $\omega\in\Omega$ construct the function $$\tilde{q_\omega}(x)=\sum\limits_{n\in I}\omega (n) f_n (x)$$
 In the expression \ref{A} set $$q(x)=q_\omega(x)=v(x)+\tilde{q_\omega}(x)$$
where $v$ and $f_n$ are measurable locally $L_1$ functions. Denote the corresponding operator $H$ in \ref{B} by $H_\omega$.

\begin{remark}
In order to have a selfadjoint operator $H_\omega$ for all $\omega\in\Omega$ some conditions have to be imposed on $q$. In case 
$a=-\infty,b=\infty$ the condition $\int_{-N}^N|q(x)|dx=O(N^3)$ as $N\rightarrow \infty$ will assure limit point case at $a$ and $b$, for
example. I shall assume that $H_\omega$ is a family of selfadjoint operators for all $\omega\in\Omega$.
\end{remark}
\begin{definition} See \cite{Kirsch}

A family  $\{A_\omega\}_{\omega\in\Omega}  $ of selfadjoint operators is measurable if for all vectors $\varphi, \psi$ the mapping 
$$\omega\rightarrow\langle f(H_\omega)\varphi,\psi\rangle$$ is measurable for any bounded Borel function $$f:\mathbb{C}\rightarrow\mathbb{C}$$

\end{definition}

\begin{remark}

If the operators $H_\omega$ defined above are selfadjoint for all $\omega\in\Omega$, the Trotter product formula
give us mesurability for the random operator $H_\omega$. This follows by an argument similar to Prop.2 of \cite{Kirsch} or Prop.V.3.1 of \cite{cl}.
In particular the functions$$\Omega\ni\omega\rightarrow\langle E_{H_\omega}(\Delta)\varphi,\psi\rangle$$ are measurable for fixed Borel sets $\Delta$
and vectors $\varphi,\psi$, where $E_{H_\omega}$ denotes the spectral projections of  $H_\omega$.
\end{remark}

Observe that the proof of the next theorem in particular implies that any fixed point is eigenvalue of
 $H_\omega$ with $\mathbb{P}$ probability zero.

\begin{theorem}\label{main}
Let $H_\omega$ be a family of selfadjoint operators in $L_2(a,b)$ defined as above $$H_\omega =- {d^2\over dx} +v(x) +\sum_{n\in I\subseteq\mathbb{Z}}\omega(n)f_n(x)$$ For a given interval     $[c,d]\subset(a,b)$  assume there is $n_0\in I$ such that $f_{n_0}(x)>0$ for  almost all $x\in(c,d)$ and $f_{n_0}(x)=0$ if 
$x\notin(c,d)$. Suppose the probability distribution $p_{n_0}$ of the random variable $\omega(n_0)$ is continuous ($p_{n_0}(\{r\})=0$ for any $r\in\mathbb{R}$).
Let $J_\omega$ be any family of measurable selfadjoint operators such that $\sigma_p(J_\omega)$ is independent of $\omega(n_0)$. Then 
$$ \mathbb{P}(\{\omega\in\Omega:\sigma_p(J_\omega)\cap \sigma_p(H_\omega)\not=\emptyset\})=0$$

\begin{proof}
 Fix any $\varphi\in L_2(a,b)$ and define $\mu_{\omega \varphi}(\Delta)=\langle E_{H_\omega}(\Delta)\varphi,\varphi\rangle$, where 
 as before $E_{H_\omega}$ denotes the spectral projection of $H_\omega$ and $\Delta$ is any Borel set. The function
$$\omega\rightarrow\mu_{\omega \varphi}(\sigma_p(J_\omega))$$ is measurable, see Corollary 3.1 \cite{drs}. This is consequence of a remarkable result about the existence of a measurable enumeration
of eigenvalues \cite{gordon}. Applying Fubini's theorem we get 
\begin{equation}\label{C}
\int_\Omega \mu_{\omega \varphi}(\sigma_p(J_\omega))d\mathbb{P} =\int\limits_{\mathbb{R}^{I\backslash \{n_0\}}}\!\!\!\!\!d\mathbb{P}(\tilde\omega)\!\!\int\limits_\mathbb{R}\mu_{\omega\varphi}(\sigma_p(J_\omega))dp_{n_0}(\omega(n_0))
\end{equation} where $\tilde \omega=\sum\limits_{n\in{I\backslash \{n_0\}}} \omega (n)\delta (n) $. Since the measure $p_{n_0}$ is continuous and $\mu_{\omega \varphi}(\{r\})>0$ implies that $r$ is an eigenvalue of $H_\omega$ (with eigenvector $E_{H_\omega}(\{r\})\varphi $), from  theorem \ref{thm1}
it follows that$$\int\limits_\mathbb{R}\mu_{\omega \varphi}(\{r\})dp_{n_0}(\omega(n_0))=0$$ for any fixed $r\in \mathbb{R}$. Therefore, if $\sigma_p(J_\omega)=\cup_{i=1}^\infty r_i(\omega )$ then
\begin{multline*}
\int\limits_\mathbb{R}\mu_{\omega\varphi}(\sigma_p(J_\omega))dp_{n_0}(\omega(n_0))=
\int\limits_\mathbb{R}\mu_{\omega\varphi}(\cup_{i=1}^\infty r_i(\omega ))dp_{n_0}(\omega(n_0))\leq \\
\int\limits_\mathbb{R}\sum_{i=1}^{\infty} \mu_{\omega\varphi}( r_i(\omega ))dp_{n_0}(\omega(n_0))=\sum_{i=1}^{\infty}\int\limits_\mathbb{R}\mu_{\omega \varphi}(\{r_i(\omega )\})dp_{n_0}(\omega(n_0))=0 
\end{multline*}
 Recall that $\sigma_p(J_\omega)$ does not depend on $\omega(n_0)$.
 Hence the expression in \ref{C} equals zero and 
\begin{equation}\label{D}
\mu_{\omega\varphi}(\sigma_p(J_\omega))=0
\end{equation}
 for a.e. $\omega \in\Omega $.
 
  Now assume that $H_\omega$ has simple spectrum for all $\omega \in\Omega $. This happens if lpc holds at most at one of the end points $a,b$. Denote by $g_\omega $ a generating vector corresponding to $H_\omega$ and let
  $\{g_i\}_{i=1}^\infty$ be an orthonormal basis of $L_2(a,b)$. Then $g_\omega=\sum_{i=1}^\infty c_i(\omega )g_i$ and using \ref{D} we get
  $$\mu_{\omega g_\omega}(\sigma_p(J_\omega))=\langle E_{H_\omega}(\sigma_p(J_\omega))g_\omega,g_\omega\rangle=
    \sum_{i,j}\overline{ c_i(\omega )}\ c_j(\omega )\langle E_{H_\omega}(\sigma_p(J_\omega))g_i,g_j\rangle  =0$$for a.e. $\omega $.
    
 Since $g_\omega $ is a generating vector, a point $r\in\mathbb{R}$ is an eigenvalue of  $H_\omega$ if and only if 
$\mu_{\omega g_\omega}(\{r\})> 0 $. Thus the theorem follows in case $H_\omega$ has simple spectrum.

In case $H_\omega$ has multiplicity two, (the only other possible case), assume $\{g_{\omega 1}, g_{\omega 2} \}$
is a generating basis (see \cite{ak} for these concepts) . From  \ref{D} we have 
$\mu_{\omega g_j}(\sigma_p(J_\omega))=0$ for $j=1,2$. Since a point $r\in\mathbb{R}$ is eigenvalue of $H_\omega$ if and only if $\mu_{\omega g_j}(\{r\})>0$ for $j=1$ or $2$, the conclusion of the theorem follows.
\end{proof}
\end{theorem}
\begin{remark}\label{frem} The result can be generalized. Instead of $\sigma_p(J_\omega)$ we could take $\cup_i r_i(\omega )$ where $r_i(\omega )$ are measurable functions which do not depend $\omega(n_0)$. In particular we can take just one function $r$ with this property  and get that $r(\omega )$ is not eigenvalue of  $H_\omega$ almost surely.

\end{remark}

\section{Applications}\label{appl}
 I mention briefly some applications.
Let us consider the operator 
\begin{equation}\label{Homega}H_\omega =- {d^2\over dx} +v(x) +\sum_{n\in I\subseteq\mathbb{Z}}\omega(n)f_n(x)
\end{equation} Assume  there is an infinite set $A\subset I$ such that the sequence of measurable functions $\{f_{n_i}\}_{i\in A}$
satisfy $$f_{n_i}(x)\left\{\begin{aligned} &>0 \quad\mbox{for almost all}\quad x\in (c_i,d_i) \\&=0\quad\mbox{otherwise}\end{aligned}\right.$$
where $\{(c_i,d_i): i\in\mathbb{N}\}$ is a collection of disjoint intervals in $(a,b)$. Assume the probability distributions $\{p_{n_i}\}_{i\in A}$
introduced in section \ref{main result} are continuous. Take a subinterval $(\tilde a,\tilde b)\subset (a,b)$
which may be bounded or unbounded, such that $(c_i,d_i)\cap(\tilde a,\tilde b)=\emptyset$ for some $i\in\mathbb{N}$ and define the operator $H_\omega$ 
as in \ref{Homega} in the space $L_2(\tilde a,\tilde b)$. Let this operator be $J_\omega $ in theorem \ref{main}. Then this theorem can be applied  and we conclude that the
 operator $H_\omega$ defined in $L_2(a,b)$ does not share eigenvalues with the same operator on $L_2(\tilde a,\tilde b)$, almost always.

 
 Let us consider another application  to the same operator \ref{Homega}. Fix $n \in I$  
 and take  $r(\omega )=h(\omega (n))$ in Remark \ref{frem}, where $h$ is a real valued measurable function.
 Then $\mathbb{P} (\{\omega \in \Omega :h(\omega (n))\in \sigma_p(H_\omega)\})=0$.  That is, almost surely $h(\omega (n))$ is not an eigenvalue of $H_\omega$, for any fixed n.
 We could take $h$ as the identity, for example.
\begin{acknowledgments}
I am grateful to Prof. D. Damanik for pointing out reference \cite{gordon}.
\end{acknowledgments}



\end{document}